\documentclass{amsart}
\usepackage{amssymb,amsthm}
\usepackage{color}
\newtheorem{theorem}{Theorem}
\newtheorem{lemma}[theorem]{Lemma}
\newtheorem{proposition}[theorem]{Proposition}
\newtheorem{definition}[theorem]{Definition}

\author[D.~Bubboloni]{Daniela Bubboloni}
\address{Daniela Bubboloni, Dipartimento di Matematica per le Decisioni,
University of Firenze,\newline Via delle Pandette, 9-D6, 50127 Firenze, Italy}
\email{daniela.bubboloni@unifi.it}

\author[F.~Luca]{Florian Luca}
\address{Florian Luca,
Instituto de Matem\'aticas de la UNAM\\
Campus Morelia, Apartado Postal 61--3 (Xangari), CP 58 089,
Morelia, Michoac\'an, Mexico}
\email{fluca@matmor.unam.mx}

\author[P. Spiga]{Pablo Spiga}
\address{Pablo Spiga,
Dipartimento di Matematica e Applicazioni, University of Milano-Bicocca,\newline
Via Cozzi 53, 20125 Milano, Italy}
\email{pablo.spiga@unimib.it}

\thanks{Address correspondence to P. Spiga,
E-mail: pablo.spiga@unimib.it\\
The first author is supported by the MIUR project ``Teoria dei gruppi
ed applicazioni''. The third author is partially supported by the University of Western Australia
as part of the  Federation Fellowship project FF0776186.}


\begin{document}

\title{Compositions of $n$ satisfying some coprimality conditions}

\begin{abstract}

An $\ell$-composition of $n$ is a sequence of length $\ell$ of positive integers summing up to $n$. In this paper, we investigate the number of $\ell$-compositions of $n$ satisfying two natural coprimality conditions. Namely, we first give an exact asymptotic formula for the number of $\ell$-compositions  having the first summand coprime to the others. Then, we estimate the number of $\ell$-compositions whose summands are all pairwise coprime.
\end{abstract}

\maketitle

\pagenumbering{arabic}

\section{Introduction}

Given a positive integer $n\in \mathbb{N}$, in this paper we are interested on the size of two sets of \emph{compositions} of $n$ both satisfying some natural coprimality conditions. For $k\geq 1$, the first set consists of the $(k+1)$-compositions $(x,y_1,\ldots,y_k)$ of $n$ with $x$ coprime to $y_i$, for each $i\in \{1,\ldots,k\}$. We denote this set by $\mathcal{A}_k(n)$ and its size by $A_k(n)$, that is,

\begin{equation*}
\begin{aligned}
\mathcal{A}_k(n)&=\{(x,y_1,\dots,y_k)\in  \mathbb{N}^{k+1}:\ n=x+y_1+\cdots+y_k,\  \gcd(x,y_1\cdots y_k)=1\},\\
A_k(n)&=\#\mathcal{A}_k(n).
\end{aligned}
\end{equation*}
Observe that $\mathcal{A}_k(n)=\varnothing$ when $n<k+1$ and that $\mathcal{A}_k(n)$ is a singleton if $n=k+1.$ Thus, we will assume that $n>k+1.$ In particular if $k\geq 2$ we will assume that $n\geq 4.$

For $k\geq 2$, the second set consists of the $k$-\emph{compositions} $(x_1,\ldots,x_{k})$ of $n$ with $x_i$ coprime to $x_j$, for every two distinct elements $i,j\in \{1,\ldots,k\}$. We denote this set by $\mathcal{B}_k(n)$ and its size by $B_k(n)$, that is,
\begin{equation*}
\begin{aligned}
\mathcal{B}_k(n)&=\{(x_1,\ldots,x_k)\in  \mathbb{N}^{k}: n=x_1+\cdots+x_k,\  \gcd(x_i,x_j)=1,~1\leq i< j\leq k\},\\
B_k(n)&=\#\mathcal{B}_k(n).
\end{aligned}
\end{equation*}
Since for $n=k$, the set $\mathcal{B}_k(n)$ a singleton, dealing with $\mathcal{B}_k(n)$ we will assume that $n>k.$
Our goal is to give an exact asymptotic estimate for $A_k(n)$ and $B_k(n)$, as functions of $n$ and $k$. We clearly have $A_1(n)=B_2(n)=\varphi(n)$ (the Euler totient function) and the asymptotic behaviour of $\varphi(n)$ is well-understood. Before stating our main results we need the following definition. Throughout the paper, we use $p$ and $q$ for primes.

\begin{definition}
\label{def}
For positive integers $k$ and $n$, define
\begin{equation*}
\begin{aligned}
&\psi_k(x)=\frac{x^k-(x-1)^k+(-1)^k}{x},
&&\delta_k(x)=\frac{(x-1)^{k}+k(x-1)^{k-1}+(-1)^{k}(k-1)}{x},\\
&C_k=\prod_{p} \left(1-\frac{\psi_k(p)}{p^{k}}\right),
&&D_k=\prod_{p}\frac{\delta_k(p)}{p^{k-1}},\\
&f_k(n)=\prod_{p\mid n} \left(1+\frac{(-1)^k}{p^k-\psi_k(p)}\right),
&&g_k(n)=\prod_{p\mid n}\left(1+\frac{(-1)^{k-1}(k-1)}{\delta_k(p)}\right).
\end{aligned}
\end{equation*}
\end{definition}
Our first main result is the following.

\begin{theorem}
\label{thm:1}
For $k\geq 1$, we have the estimate
$$
\left |A_k(n)-C_k\,f_k(n)\,\frac{n^{k}}{k!}\right|\le \frac{2+e}{\sqrt{2\pi k}}(e^2\log n)^{k}n^{k-1}.
$$
\end{theorem}

In Theorem~\ref{thm:1} and in what follows, we could use the Landau symbol $O$ with its usual meaning. However, we usually shall avoid the symbol $O$ because
we want our estimates to be completely explicit. Throughout the proofs we shall use $\theta$ (with or without subscripts) for a real number with $|\theta|\le 1$.

In view of Theorem~\ref{thm:1} we have that the leading term of $A_k(n)$ is $n^k/k!$ multiplied by $C_k$ (which depends only on $k$) and by $f_k(n)$ (which depends upon the prime factorization of $n$).

When $k=1,$ since $\psi_1(x)=0$, we have that $C_1=1$ and
\begin{equation*}
\label{varphiC}
C_1f_1(n)n=\prod_{p\mid n}\left(1-\frac{1}{p}\right)n=\varphi(n).
\end{equation*}
So, the leading term in Theorem~\ref{thm:1} actually equals $A_1(n)$.

Our next result collects some information on $C_k$ and on $f_k(n)$ which, together with Theorem~\ref{thm:1}, unravels the asymptotic behaviour of $A_k(n)$.

\begin{theorem}
\label{thm:2}
For every $k\geq 2$, the series $C_k$ converges and $0<C_k<1.$ Furthermore $2/3<f_k(n)<2$.
\end{theorem}

For $B_k(n)$, we prove the following.

\begin{theorem}
\label{thm:allcoprime}
For $k\geq 2$ and $n\geq e^{k2^{k+2}}$, we have the estimate
$$
\left|B_k(n)-D_k g_k(n) \frac{n^{k-1}}{(k-1)!}\right|\le \frac{707n^{k-1}}{\log n}.
$$
\end{theorem}

Exactly as in Theorem~\ref{thm:1}, we see that the leading term of $B_k(n)$ is $n^{k-1}/(k-1)!$ multiplied by $D_k$ (which depends only on $k$) and by $g_k(n)$ (which depends upon the prime factorization of $n$).

When $k=2,$ since $\delta_2(x)=x$, we have that $D_2=1$ and
\begin{equation*}\label{varphiD}D_2g_2(n)n=\prod_{p\mid n}\left(1-\frac{1}{p}\right)n=\varphi(n).
\end{equation*}
So the leading term in Theorem~\ref{thm:allcoprime} actually equals $B_2(n)$.

Theorem~\ref{thm:3} collects some information on $D_k$ and $g_k(n)$, which helps to describe the order of magnitude of $B_k(n)$.

\begin{theorem}
\label{thm:3}
For every $k\geq 3$, the series $D_k$ converges and $0<D_k< 1.$ Furthermore $1/2k<g_k(n)<2k$.
\end{theorem}

Similar problems on compositions with restricted arithmetical conditions
have been studied in~\cite{GO} and~\cite{Lis}. In particular, using the principle of inclusion-exclusion, Gould~\cite[Theorem~5]{GO} has obtained a formula for the number of $k$-compositions $(x_1,\ldots,x_k)$ of $n$ with $\gcd(x_1,x_2,\ldots,x_k)=1$.

Finally, in Table~\ref{table}, we give some approximate values for $C_k$ and $D_k$ (for $1\leq k\leq 7$), which are obtained with the help of \texttt{magma}~\cite{magma}.

\begin{table}[!h]
\begin{tabular}{|l|l|l|}\hline
$k$&$C_k$&$D_k$\\\hline
$1$&$1$&--\\
$2$&$0.32263$&$1$\\
$3$&$0.38159$&$0.12548$\\
$4$&$0.26778$&$0.19680$\\
$5$&$0.26328$&$0.01312$\\
$6$&$0.23051$&$0.02330$\\
$7$&$0.22123$&$0.00099$\\\hline
\end{tabular}\label{table}
\smallskip
\caption{Some values for $C_k$ and $D_k$}
\end{table}

\subsection{Applications to Group Theory and to Galois Theory}
In \cite{BP}, the first author together with  Praeger, investigated the {\em normal coverings} of a finite group $G$, that is, the families $H_1,\ldots,H_r$ of proper subgroups of $G$ such that each element of  $G$ has a conjugate in  $H_i$, for some $i\in \{1,\ldots,r\}$. The minimum $r$ is usually denoted by $\gamma(G)$. They find that when $G$ is the symmetric group $S_n$ or the alternating group $A_n$, the number
$\gamma(G)$ lies between $a\,\varphi(n)$ and $b\,n$ for certain positive constants
$a$ and $b$. 
More recently, Bubboloni, Spiga and  Praeger~\cite {BPS} have developed some new research on this topic starting with the idea that  primitive subgroups of the symmetric group are ``few and small''  (see \cite{B}, \cite{LS},  \cite{LMS} and \cite{MA}) and therefore cannot  play a significant role in normal coverings.  With an application of Theorem~\ref{thm:1},  they show that, for $G=S_n$ or $ A_n $, the number $\gamma(G)$ is asymptotically linear in $n$ (improving  every previous result in this area).

The normal coverings of the symmetric and of the alternating group are relevant for some problems in  Galois theory~\cite{BP}.
Let $f(x)\in \mathbb{Z}[x]$ be a polynomial which has a root $\mod\,p,$ for all primes $p$, and consider its Galois group over the rationals $G=Gal_{\mathbb{Q}}(f)$.  Let $f_1(x),\dots, f_k(x)\in \mathbb{Z}[x]$ be the
distinct irreducible factors of $f(x)$ over $\mathbb{Q}$, and suppose that no $f_i$ is linear. By \cite[Theorem 2]{BB},
we have $k\geq \gamma (G).$
In other words, for a polynomial $f(x)$ which has a root $\mod\,p,$ for all primes $p$, but no root in $\mathbb{Q}$, the number of subgroups involved in a minimal normal covering of its Galois group is a lower bound for the number of distinct irreducible factors of $f(x)$ over $\mathbb{Q}$. In this context the pertinence of the results in \cite{BP}, in this paper and in \cite{BPS} relies on the fact that the most common Galois groups are the symmetric and the alternating groups~\cite{WA}.

Finally, we point out that Theorems~\ref{thm:1} and~\ref{thm:2} are also used in~\cite{BS}, to obtain some bounds on the diameter of the generating graph of $S^n$, for $n\geq 1$ and for a finite non-abelian simple group $S$.

\subsection{Structure of the paper} Theorems~\ref{thm:2} and~\ref{thm:3} are proved in Section~\ref{sec:2}, Theorem~\ref{thm:1} is proved in Section~\ref{sec:proof} and Theorem~\ref{thm:allcoprime} is proved in Section~\ref{sec:allcoprime}.

\section{En route to the proof of Theorem \ref{thm:1}}
We denote with
\begin{equation*}
\begin{aligned}
\mathcal{K}_k(n)=&\{(x_1,x_2,\dots,x_{k+1})\in\mathbb{N}^{k+1} :\  n=\sum_{i=1}^{k+1}x_i\} \textrm{ and}\\
\mathcal{U}_k(n)=&\{(x_1,x_2,\dots,x_{k+1})\in(\mathbb{N}\cup\{0\})^{k+1} :\ n=\sum_{i=1}^{k+1}x_i\},
\end{aligned}
\end{equation*}
respectively, the set of  $(k+1)$-compositions and  the set of {\em generalized $(k+1)$-compositions} of $n$. It is well known that
\begin{equation}
\begin{aligned}\label{co}
K_k(n)=&\#\mathcal{K}_k(n)=\binom{n-1}{k}=\frac{n^{k}}{k!}+\theta_K kn^{k-1}, \textrm{ and}\\
U_k(n)=&\#\mathcal{U}_k(n)
=\binom{n+k}{k}=\frac{n^{k}}{k!}+ \theta_U kn^{k-1}
\end{aligned}
\end{equation}
(see for instance~\cite{GO}).

The following definition will turn out to be crucial in the proof of Theorem~\ref{thm:1}.

\begin{definition}
\label{def:kdn}
For a square-free positive integer $d\geq 1$, write $$\mathcal{K}_{k,d}(n)=\{(x,y_1,\dots,y_k  )\in  \mathcal{K}_k(n):\ d \ \text{divides}\ \gcd(x,\prod_{i=1}^ky_i)\}$$ and $$
K_{k,d}(n)=\# \mathcal{K}_{k,d}(n).$$
\end{definition}

\noindent Note that, if $J=\{p_1,\ldots,p_s\}$  is a set of primes and $d=p_1\cdots p_s$, then
$$
\mathcal{K}_{k,d}(n)=\bigcap_{i=1}^{s}\mathcal{K}_{k,p_i}(n).
$$
Moreover $K_{k,1}(n)=K_{k}(n).$
Clearly $\mathcal{K}_{k,d}(n)$ is empty when $d>n.$

\noindent Our idea to compute $A_k(n)$ is to use the principle of inclusion-exclusion  as
$$
\mathcal{A}_{k}(n)=\mathcal{K}_{k}(n)\setminus \bigcup_{p\in \mathbb{P}_n}\mathcal{K}_{k,p}(n),
$$
where
$$
\mathbb{P}_n=\{r \in \mathbb{N}\ :\ r \ \text{prime}, \ r\leq n \}.
$$
Namely,
\begin{eqnarray}\nonumber
A_{k}(n)&=&\#\mathcal{K}_{k}(n)-\#\left(\bigcup_{p \in \mathbb{P}_n}\mathcal{K}_{k,p}(n)\right)=K_{k}(n)+\sum_{\emptyset\neq J \subseteq \mathbb{P}_n} (-1)^{\#J}\#\bigcap_{p\in J}\mathcal{K}_{k,p}(n)\\
&=&\sum_{J\subseteq \mathbb{P}_n }(-1)^{\#J}\#\bigcap_{p\in J}\mathcal{K}_{k,p}(n)=\sum_{1\leq d\leq n} \mu(d) K_{k,d}(n),\label{eq:incexc}
\end{eqnarray}
where $\mu$ is the M\"obius function.

In light of \eqref{eq:incexc}, to prove Theorem \ref{thm:1} we need to estimate the numbers $K_{k,d}(n)$.   This will be possible thanks to some lemmas on linear equations modulo $p$ which we give in Section \ref{sec:2}. An asymptotic formula for $K_{k,d}(n)$ is then obtained in Proposition~\ref{J}. Throughout the rest of this paper, we reserve the letter $d$ to denote a square-free positive integer.

\section{Linear equations modulo $p$ and the proof of Theorems~\ref{thm:2} and~\ref{thm:3}}\label{sec:2}

We start by  introducing two auxiliary polynomials $\phi_k(x)$ and $\eta_k(x)$ which are closely related to $\psi_k(x)$ and $\delta_k(x)$ in Definition~\ref{def}. These polynomials turn out to be fundamental
for understanding the {\it local} aspects of the sets  $\mathcal{A}_{k}(n)$ and $\mathcal{B}_k(n)$.

\begin{definition}\label{phi}
For $k\geq 1$, define
\begin{equation*}
\begin{aligned}
\phi_k(x)&=\frac{x^k-(x-1)^k+(-1)^{k+1}(x-1)}{x},\\
\eta_k(x)&=\frac{(x-1)^k+k(x-1)^{k-1}+(-1)^{k-1}(k-1)(x-1)}{x}.
\end{aligned}
\end{equation*}
\end{definition}
A direct calculation shows immediately that for any $k\geq 1,$ we have:
\begin{equation}\label{etadelta}
\begin{aligned}
\phi_k(x)&=\psi_k(x)+(-1)^{k+1},\\
\eta_k(x)&=\delta_k(x)+(-1)^{k-1}(k-1).
\end{aligned}
\end{equation}
When $k=1$, we have
\begin{equation*}\label{1}
\psi_1(x)=0,\quad\delta_1(x)=1,\quad \phi_1(x)=1,\quad \eta_1(x)=1.
\end{equation*}

\begin{lemma}
\label{modp}
Let $k$ and $n$ be integers with $k\geq 1$ and $n\geq 0$. Write
\begin{equation*}
\begin{aligned}
S_k(x)&=x^{k-1}-\phi_k(x)=\frac{(x-1)^k+(-1)^k(x-1)}{x},\\
W_k(x)&=x^{k-1}-\psi_k(x)= \frac{(x-1)^k+(-1)^{k+1}}{x}.
\end{aligned}
\end{equation*}
Then:
\begin{itemize}
 \item [a)] The number of solutions of
\begin{equation}
\label{homA}y_1^*+\cdots +y_{k}^*\,\equiv n \pmod p,\end{equation}
with $y_i^*\in \{0,\ldots,p-1\}$ for every $i\in\{1,\ldots,k\}$ and with $y_j^*=0$ for some $j\in \{1,\ldots,k\}$, is  $\phi_k(p)$ if $p\mid n$ and $\psi_k(p)$ if $p\nmid n$.

\item[b)]For $k\geq 2$, the inequality
$$
\max\{\phi_k(x),\psi_k(x)\}\le kx^{k-2}
$$
holds for all $x\ge 2$.

\item[c)]The number of solutions of~\eqref{homA}, with $y_j^*\in \{1,\ldots,p-1\}$ for every $i\in\{1,\ldots,k\}$, is $S_k(p)$ if $p\mid n$, and $W_k(p)$ if $p\nmid n$.

\item[d)]The number of solutions of~\eqref{homA}, with $y_j^*\in \{0,\ldots,p-1\}$ for every $i\in\{1,\ldots,k\}$ and with $y_j^*=0$ for at most one $j\in \{1,\ldots,k\}$, is $\eta_k(p)$ if $p\mid n$, and  $\delta_k(p)$ if $p\nmid n$.

\item[e)] For  $k\geq 1,$ we have
$$
\max\{\psi_k(p),\,\delta_k(p),\,\phi_k(p),\, \eta_k(p)\}\leq p^{k-1}.
$$
Moreover, if $k\geq 3$, then
$$
(p-1)^{k-1}\leq \delta_k(p)<p^{k-1}.
$$

\item[f)] For $k\geq 2,$ we have
$$
\min\{\psi_k(p),\,\delta_k(p),\,\phi_k(p),\, \eta_k(p)\}\geq 1.
$$

\item[g)]For $k\geq 3$, the inequality
$$
\max\left\{\frac{p^{k-1}}{\eta_k(p)},\frac{p^{k-1}}{\delta_k(p)}\right\}\leq 1+\frac{k2^k}{p^2}
$$
holds for all $p\geq \sqrt{k2^k}$.
\end{itemize}
\end{lemma}

\begin{proof} a) We prove a) by induction on $k$. If $k=1$, then the result is obvious because $\phi_1(p)=1$ and $\psi_1(p)=0$. Assume that $k\geq 2$. Suppose that $p\mid n$. Now ~\eqref{homA} has $p^{k-2}$ solutions with $y_1^*=0$. Also, for any $y_1^*\neq 0$, ~\eqref{homA} is equivalent to $y_2^*+\cdots +y_k^*\equiv (-y_1^*)\pmod p$ and so, by induction, there exist $\psi_{k-1}(p)$ solutions of ~\eqref{homA} having at least one coordinate being zero and with a fixed $y_1^*\neq 0$.  Summing up, the number of solutions of ~\eqref{homA} with at least one coordinate being zero is
\begin{eqnarray*}
p^{k-2}+(p-1)\psi_{k-1}(p)&=&p^{k-1}+(p-1)\frac{p^{k-1}-(p-1)^{k-1}+(-1)^{k-1}}{p}\\
&=&\frac{p^k-(p-1)^k+(-1)^{k+1}(p-1)}{p}=\phi_k(p).
\end{eqnarray*}
Suppose that $p\nmid n$. The number of solutions of ~\eqref{homA} with $y_1^*=0$ is $p^{k-2}$. If $y_1^*\in\{0,\ldots,p-1\}$ and $y_1^*\equiv n\pmod p$, then ~\eqref{homA} is equivalent to $y_2^*+\cdots+y_k^*\equiv 0\pmod p$, which, by induction, has $\phi_{k-1}(p)$ solutions with at least one coordinate being zero. Finally, for any $y_1^*\not\equiv 0,n\pmod p$, ~\eqref{homA} is equivalent to $y_2^*+\cdots +y_k^*\equiv (n-y_1^*)\pmod p$ which, again by induction, has $\psi_{k-1}(p)$ solutions with at least one coordinate being zero. Summing up, the number of solutions of ~\eqref{homA} with at least one coordinate being zero is
\begin{eqnarray*}
p^{k-2}+\phi_{k-1}(p)+(p-2)\psi_{k-1}(p)&=&\frac{p^k-(p-1)^k+(-1)^{k}}{p}=\psi_k(p).
\end{eqnarray*}

b) If $k=2$, we have $ \phi_2(x)=1<2$ and $\psi_2(x)=2\le 2$. Now assume that $k\geq 3$. From the factorization, $u^k-v^k=(u-v)(u^{k-1}+u^{k-2}v+\cdots +uv^{k-2}+v^{k-1})$, we see that
\begin{eqnarray*}
\phi_k(x) & = & \frac{1}{x}\left(x^{k-1}+x^{k-2}(x-1)+\cdots+(x-1)^{k-1}+(-1)^{k+1}(x-1)\right),\\
\psi_k(x) & = & \frac{1}{x} \left(x^{k-1}+x^{k-2}(x-1)+\cdots +(x-1)^{k-1}+(-1)^{k}\right).
\end{eqnarray*}
For each $i\in \{0,\ldots,k-1\}$ with $i\neq k-2$, we have $x^{i}(x-1)^{k-1-i}\leq x^{k-1}$ while $x^{k-2}(x-1)+1\leq x^{k-2}(x-1)+(x-1)\leq x^{k-1}$, for all $x\ge 2$. The result now follows.

c) For any $n\geq 0,$ the set  $\mathcal{L}_k(n)$ of solutions in $\{0,\ldots,p-1\}$ of  the linear congruence ~\eqref{homA} has size $p^{k-1}$. Moreover, the solutions of~\eqref{homA} with no coordinate being zero is the complement of the solutions described in a), with respect to $\mathcal{L}_k(n)$,
 that is,  the number of solutions of ~\eqref{homA} with no coordinate being zero is
\begin{eqnarray*}
S_k(p)= p^{k-1}-\phi_k(p) \quad & {\text{\rm if}} &  p\mid n;\\
W_k(p)= p^{k-1}-\psi_k(p)\quad & {\text{\rm if}} & p\nmid n.
\end{eqnarray*}

d) The solutions of  ~\eqref{homA} with at most one coordinate being zero is the disjoint union of the
solutions of  ~\eqref{homA} with no coordinate being zero and the  solutions of  ~\eqref{homA} with exactly one coordinate being zero. Thus, by c), we get that the number of solutions of  ~\eqref{homA} with at most one coordinate being zero is
\begin{eqnarray*}
S_k(p)+kS_{k-1}(p)=\eta_k(p)\quad  & {\text{\rm if}} &  p\mid n;\\
W_k(p)+kW_{k-1}(p)=\delta_k(p)\quad & {\text{\rm  if}} & p\nmid n.
\end{eqnarray*}

e) For any $n\geq 0,$  the linear congruence ~\eqref{homA} has exactly $p^{k-1}$ solutions in $\{0,\ldots,p-1\}$, which are obtained choosing freely the values of $k-1$ variables $y_i^*$ and computing the last one.
Moreover,  for all $p$, by a) and d), we can interpret  $\psi_k(p),\,\delta_k(p)$ as  counting the number of  particular solutions of ~\eqref{homA} with $n=1$ and $\phi_k(p),\, \eta_k(p)$ as counting the number of particular solutions of ~\eqref{homA} with $n=0.$ This gives for any $k\geq 1, $
$$
\max\{\psi_k(p),\,\delta_k(p),\,\phi_k(p),\, \eta_k(p)\}\leq p^{k-1}.
$$
Assume now $k\geq 3$ and observe that, by  d), $\delta_k(p)$ is the number of solutions of~\eqref{homA} with at most one $y_j^*=0$ and  with $n=1.$ Among these solutions we find those obtained by selecting arbitrarily $y_2^*, \dots, y_k^*\in \{1,\dots,p-1\}$ and determining the corresponding $y_1^*,$ which says $\delta_k(p)\geq (p-1)^{k-1}.$ On the other hand, since we cannot assign $0$ in two of the $k-1\geq 2$ variables $y_2^*, \dots, y_k^*$, we have $\delta_k(p)<p^{k-1}.$

f) To get
$$
\min\{\psi_k(p),\,\delta_k(p),\,\phi_k(p),\, \eta_k(p)\}\geq 1,
$$
observe that there exists at least one solution for those equations when $k\geq 2.$

g) We begin showing that, for all primes $p$ and $k\geq 3,$ the inequality
\begin{equation}
\label{max}
\max\{p^{k-1}-\delta_k(p),\, p^{k-1}-\eta_k(p)\}\leq k2^{k-1}p^{k-3}
\end{equation}
holds.
By expanding the terms in $\delta_k(p)$, we get
\begin{eqnarray*}
\delta_k(p)&=&\sum_{i=1}^k(-1)^{k-i}\binom{k}{i}p^{i-1}+k\sum_{i=1}^{k-1}(-1)^{k-1-i}\binom{k-1}{i}p^{i-1}\\
&=&p^{k-1}-\sum_{i=1}^{k-2}(-1)^{k-i}\left[k\binom{k-1}{i}-\binom{k}{i}\right]p^{i-1}.
\end{eqnarray*}
Hence, by e),
\begin{equation}\label{delta}
0\leq p^{k-1}-\delta_k(p)= \sum_{i=1}^{k-2}(-1)^{k-i}\left[k\binom{k-1}{i}-\binom{k}{i}\right]p^{i-1}
\end{equation}
and
\begin{equation}\label{eta}
0\leq p^{k-1}-\eta_k(p)= p^{k-1}-\delta_k(p)+(-1)^k(k-1).
\end{equation}

Note that since
$$
k\binom{k-1}{i}=(k-i)\binom{k}{i}\ge \binom{k}{i}\qquad {\text{\rm for~all}}\quad i\in \{1,\ldots,k-2\},
$$
we have $$k\binom{k-1}{i}-\binom{k}{i}\geq 0.$$

If $k$ is odd, we have from \eqref{delta}
$$
p^{k-1}-\delta_k(p)\leq  \sum_{i=1}^{k-2}k\binom{k-1}{i}p^{k-3}\leq k2^{k-1}p^{k-3}.
$$
By~\eqref{eta}, we have $ 
p^{k-1}-\eta_k(p)\leq p^{k-1}-\delta_k(p),$  and so 
the same inequality holds for $p^{k-1}-\eta_k(p).$

If $k$ is even, then $k-1$ is odd, so that the term corresponding to the choice $i=1$ in the sum in \eqref{delta} is negative; moreover, since $k-2\geq 2$, there is at least another term in the sum in \eqref{delta}. It follows, by \eqref{eta},  that
\begin{eqnarray*}
p^{k-1}-\eta_k(p)&=&p^{k-1}-\delta_k(p)+k-1\leq \sum_{i=2}^{k-2}k\binom{k-1}{i}p^{k-3}+k-1\\
&=&k(2^{k-1}-k-1)p^{k-3}+k-1\leq k2^{k-1}p^{k-3},
\end{eqnarray*}
because $k(k+1)p^{k-3}\geq k-1$ for any $k\geq 3.$ The same conclusion follows also for $p^{k-1}-\delta_k(p)$ since
$
p^{k-1}-\delta_k(p)\leq p^{k-1}-\eta_k(p).
$
So, we have proved~\eqref{max}.

Therefore we can write
$$
\frac{\delta_k(p)}{p^{k-1}}=1-\frac{\delta_k'(p)}{p^{2}},\qquad {\text{\rm with}}\quad 0\leq \delta_k'(p)\le k2^{k-1},
$$
as well as
$$
\frac{\eta_k(p)}{p^{k-1}}=1-\frac{\eta_k'(p)}{p^{2}},\qquad {\text{\rm with}}\quad 0\leq\eta_k'(p)\le k2^{k-1}.
$$
Thus,
\begin{eqnarray*}
\frac{p^{k-1}}{\delta_k(p)}&=&\frac{1}{1-\delta_k'(p)/p^2}\leq \frac{1}{1-k2^{k-1}/p^2}\leq 1+\frac{k2^k}{p^2}
\end{eqnarray*}
for $p^2\geq k2^k$, using the fact that
$$
\frac{1}{1-y}\leq 1+2y,\qquad {\text{\rm for~all}}\qquad 0\leq y\leq 1/2.
$$
The same argument applies to $\eta_k(p)$
and gives
$$
\max\left\{\frac{p^{k-1}}{\eta_k(p)},\frac{p^{k-1}}{\delta_k(p)}\right\}\leq 1+\frac{k2^k}{p^2}
$$
for any $k\geq 3$ and any prime $p\geq \sqrt{k2^k}$.
\end{proof}

Using Lemma~\ref{modp}, we are now ready to prove Theorems~\ref{thm:2} and~\ref{thm:3}.

\begin{proof}
[Proof of Theorem~$\ref{thm:2}$]
Let $k\geq 2$ and consider
$$
C_k=\prod_{p} \left(1-\frac{\psi_k(p)}{p^{k}}\right)=\exp\left\{-\sum_{p}-\log\left(1-\frac{\psi_k(p)}{p^{k}}\right)\right\}.
$$
By Lemma~\ref{modp} e) and f), we have
$$
1\leq \psi_k(p)\leq p^{k-1}
$$
and therefore $1/p^k\leq \psi_k(p)/p^{k}\leq 1/p$, which shows that  $0<1-\psi_k(p)/p^k<1$.  So the series
\begin{equation}
\label{series}
\sum_{p}-\log\left(1-\frac{\psi_k(p)}{p^{k}}\right)
\end{equation}
has positive terms and  $C_k$ is a real number in $[0,1).$ We will show that $C_k\neq 0$ by observing that the series \eqref{series}
converges. To do that, we expand
\begin{eqnarray*}
1-\frac{\psi_k(p)}{p^{k}}&=&1- \frac{p^k-(p-1)^k+(-1)^k}{p}\\
&=&1-\frac{k}{p^2} -\sum_{i=2}^{k-1}(-1)^{i+1}\binom{k}{i}\frac{1}{p^{i+1}}\\
&=& 1-\frac{k}{p^2} + o\left(\frac{1}{p^2}\right)\quad {\text{\rm as}}\quad p\to\infty,
\end{eqnarray*}
and therefore we obtain
$$
-\log\left(1-\frac{\psi_k(p)}{p^{k}}\right)\sim \frac{k}{p^2}\quad {\text{\rm as}}\quad p\to\infty.
$$
Since
$$\sum_{p}\frac{k}{p^2}< \sum_{n=1}^{\infty}\frac{k}{n^2}=\frac{k\pi^2}{6}
$$
converges, also the series \eqref{series} converges. 

We now turn to the inequalities involving the functions
$$
f_k(n)=\prod_{p\mid n} \left(1+\frac{(-1)^k}{p^k-\psi_k(p)}\right).
$$
By Lemma \ref{modp} e), we have
$$
p^k-\psi_k(p)\geq p^k-p^{k-1}>0
$$
and thus for any $n\in \mathbb{N}$ we get
$f_k(n)<1$ if $k$ is odd, and $f_k(n)>1$ if $k$ is even. Observe also that the function $p^k-p^{k-1}$ is increasing in $k.$
To find an upper bound when $k$ is even, note that $p^k-p^{k-1}\leq p^2-p$ and thus
$$
f_k(n)\leq \prod_{p} \left(1+\frac{1}{p^2-p}\right).
$$
It follows that
\begin{eqnarray*}
\log \left(f_k(n)\right)&\leq&\sum_{p}\log \left(1+\frac{1}{p^2-p}\right)\\
& = & \sum_{p\le 1000} \log\left(1+\frac{1}{p^2-p}\right)+
\sum_{p>1000} \log\left(1+\frac{1}{p^2-p}\right)\\
&<& 0.665+\sum_{n\ge 1001} \frac{1}{n^2-n}=0.665+0.001=0.666,
\end{eqnarray*}
which gives
$$
f_k(n)<e^{0.666}<2.
$$
Finally, we find a lower bound when $k$ is odd starting from $p^k-p^{k-1}\leq p^3-p^2,$ which immediately gives
$$
f_k(n)\geq \prod_{p} \left(1-\frac{1}{p^3-p^2}\right).
$$
It follows that
$$
\log \left(f_k(n)\right)\geq-\sum_{p}\log \left(\left(1-\frac{1}{p^3-p^2}\right)^{-1}\right),
$$
Since
\begin{eqnarray*}
\sum_{p}\log \left(\left(1-\frac{1}{p^3-p^2}\right)^{-1}\right)&=&\sum_{p}\log \left(1+\frac{1}{p^3-p^2-1}\right)\\
&  = &  \sum_{p\le 1001} \log \left(1+\frac{1}{p^3-p^2-1}\right)\\
 & + & \sum_{p>1001} \log\left(1+\frac{1}{p^3-p^2-1}\right)\\
& < & 0.361+\sum_{p> 1001} \frac{1}{p^3-p^2-1}\\
& < & 0.361+\sum_{p> 1001} \frac{1}{p(p-1)(p-2)}\\
& < & 0.361+\sum_{n\ge 1002} \left(\frac{1}{2(n-2)}-\frac{1}{n-1}+\frac{1}{2n}\right)\\
& < & 0.361+0.0005=0.3615,
\end{eqnarray*}
it follows that
$$
f_k(n)\geq e^{-0.3615}>\frac{2}{3}.
$$
\end{proof}

\begin{proof}[Proof of Theorem~\ref{thm:3}]
Let $k\geq 3$. From Lemma~\ref{modp}~e) and f), we have
$$
1\leq\delta_k(p)< p^{k-1} \quad {\text{\rm so~ that}}\quad  0<\frac{\delta_k(p)}{p^{k-1}}< 1,
$$
and
$$
D_k=\prod_{p} \frac{\delta_k(p)}{p^{k-1}}=\exp\left\{-\sum_{p}\log\left(\frac{p^{k-1}}{\delta_k(p)}\right)\right\},
$$
where the series
\begin{equation}
\label{series2}
\sum_{p}\log\left(\frac{p^{k-1}}{\delta_k(p)}\right)
\end{equation}
has positive terms. This gives immediately that $D_k\in [0,1)$ and we need only to show that $D_k\neq0.$
To do that we prove that the series \eqref{series2} converges.

For $p\geq \sqrt{k2^k}$, Lemma~\ref{modp}~g) gives
$$
\frac{p^{k-1}}{\delta_k(p)}\leq 1+\frac{k2^k}{p^2}\quad {\text{\rm and~ therefore}}\quad \log\left(\frac{p^{k-1}}{\delta_k(p)}\right)\leq \frac{k2^k}{p^2}.
$$
Thus,
$$
\sum_{p}\log\left(\frac{p^{k-1}}{\delta_k(p)}\right)\leq\sum_{p<\sqrt{k2^k}}\log\left(\frac{p^{k-1}}{\delta_k(p)}\right)+ k2^k\sum_{n=1}^{\infty}\frac{1}{n^2},
$$
where the last series converges.

We now turn to the inequalities involving the functions $g_k(n)$ for $k\geq 3$. First of all, observe that $g_k(n)<1$ if $k$ is even, and $g_k(n)>1$ if $k$ is odd, because $(k-1)/\delta_k(p)>0$.
To get some bounds for $g_k(n),$ we begin computing:
\begin{equation*}\label{g}
\delta_k(2)=\left\{
\begin{array}{cccc}
1 & \textrm{if}& k & \textrm{is odd},\\
k & \textrm{if}&  k&  \textrm{is even},\\
\end{array}\right.\\
\end{equation*}
and
\begin{equation*}
\label{delta34}
\delta_3(p)= p^2-3,\qquad \delta_4(p)= p^3-6p+8.
\end{equation*}
Recall that, by Lemma \ref{modp} e), for any $k\geq 3$, we have $
\delta_k(p)\geq (p-1)^{k-1}$. 
Let $k$ be odd. For any $p\geq 3,$ we have
$$
\frac{k-1}{\delta_k(p)}\leq \frac{2}{\delta_3(p)},\quad {\text{\rm that~ is}}\quad \delta_k(p)\geq \frac{(k-1)(p^2-3)}{2}
$$
(this is trivial when $k=3$ and, for $k\geq 5$, we have
$
(p-1)^{k-1}\geq (k-1)(p^2-3)/2$).
Similarly if $k$ is even, for any $p\geq 3,$ we have
$$
\frac{k-1}{\delta_k(p)}\leq \frac{3}{\delta_4(p)},\quad {\text{\rm that~ is}}\quad \delta_k(p)\geq \frac{(k-1)(p^3-6p+8)}{3}
$$
(this is trivial when $k=4$ and, for $k\geq 6$, we have
$
(p-1)^{k-1}\geq (k-1)(p^3-6p +8)/3
$).
It follows that when $k$ is odd
\[
g_k(n)=\prod_{p\mid n}\left(1+\frac{k-1}{\delta_k(p)}\right)\leq\left\{
\begin{array}{ccc}
k\prod_{p>2}\left(1+\frac{2}{p^2-3}\right)&\textrm{if }n \textrm{ is even},\\
\prod_{p>2}\left(1+\frac{2}{p^2-3}\right)&\textrm{if }n \textrm{ is odd}.
\end{array}\right.
\]
Similarly, for $k$ even, we obtain
\[
g_k(n)=\prod_{p\mid n}\left(1-\frac{k-1}{\delta_k(p)}\right)\geq\left\{
\begin{array}{ccc}
\frac{1}{k}\prod_{p>2}\left(1-\frac{3}{p^3-6p+8}\right)&\textrm{if }n \textrm{ is even},\\
\prod_{p>2}\left(1-\frac{3}{p^3-6p+8}\right)&\textrm{if }n \textrm{ is odd}.
\end{array}\right.
\]
It remains to give estimates for the numbers
$$
a=\prod_{p>2}\left(1+\frac{2}{p^2-3}\right)\quad\textrm{and}\quad b=\prod_{p>2}\left(1-\frac{3}{p^3-6p+8}\right) .
$$
We have
\begin{eqnarray*}
 \log a &=&\sum_{p\geq 3}\log \left(1+\frac{2}{p^2-3}\right) \leq 2\,\sum_{p\geq 3}\frac{1}{p^2-3}\\
&\leq&2\,\left[ \sum_{p=3}^{11}\frac{1}{p^2-3} +\frac{13^2}{13^2-3}\int_{13}^{\infty}\frac{dx}{x^2}  \right]<\log 2,\\
\end{eqnarray*}
and  so $a<2.$
Finally,  observe that
$$
\log b =-\sum_{p\geq 3}\log \left(1+\frac{3}{p^3-6p+5}\right)
$$
and that
\begin{eqnarray*}
\sum_{p\geq 3}\log \left(1+\frac{3}{p^3-6p+5}\right)&\leq& \sum_{p\geq 3}\frac{3}{p^3-6p+5}
\leq \frac{3}{14}+4\sum_{p\geq 5}\frac{1}{p^3}\\&\leq& \frac{3}{14}+4\left[\frac{1}{125}+
\int_{5}^{\infty}\frac{dx}{x^3}  \right]<\log 2. \\
\end{eqnarray*}
It follows that $b>1/2.$
\end{proof}

\section{The coefficients $K_{k,d}(n)$ and the proof of Theorem~\ref{thm:1}}\label{sec:proof}
Now we are ready to compute $K_{k,d}(n)$. We will show in Proposition~\ref{J} that the leading term of $K_{k,d}(n)$ is  $(n/d)^k/k!$ multiplied by the correction factor defined below.

\begin{definition}\label{theta}
For positive integers $k$ and $n$,
write
$$
\Theta_{k,n}(d)=\prod_{p\mid d,\,p\mid n} \phi_k(p)\prod_{p\mid d,\,p\nmid n} \psi_k(p),
$$
where for $d=1$ the empty product is interpreted 1.
\end{definition}

The function $\Theta_{k,n}(d)$ of the square-free number $d$ is multiplicative; that is, if $d_1$ and $d_2$ are coprime square-free numbers, then $\Theta_{k,n}(d_1d_2)=\Theta_{k,n}(d_1)\Theta_{k,n}(d_2)$.

Note also that, by  Lemma~\ref{modp}~b) and f), we have
\begin{equation}
\label{eq:PHI}
1\leq \Theta_{k,n}(d)\le k^{\omega(d)} d^{k-2}\quad \hbox{for all}\  k\geq 2,
\end{equation}
where, for a positive integer $m$, $\omega(m)$ denotes the number of distinct prime divisors of $m.$

Here and in the next section, to go straight on into computations, we need also this technical lemma.

\begin{lemma}
\label{power}
Let $x,y$ and $c\geq 1$ be real numbers and $k\in \mathbb{N}.$ If $|x-y|\le ck$,
then
$$
|x^k-y^k|< \frac{k!e^{1/c}(ce)^k}{\sqrt{2\pi k}}\,y^{k-1}.
$$
\end{lemma}

\begin{proof}
Let $x=y+\theta ck$. We then have
$$
x^k=y^k + \sum_{i=0}^{k-1}\binom{k}{i}y^{i}(ck\theta)^{k-i}
$$
and
$$
|x^k-y^k|  \le  y^{k-1}\sum_{i=0}^{k-1}\binom{k}{i}(ck)^{k-i}
 <  y^{k-1}\sum_{i=0}^{k} \binom{k}{i} (ck)^{k-i}= y^{k-1} (ck+1)^{k}.
$$
As the exponential function with base greater than $1$ is increasing, we obtain
$$(ck+1)^k =(ck)^k \left[\left(1+\frac{1}{ck}\right)^{ck}\right]^{1/c}<(ck)^k\, e^{1/c}.$$
By Stirling's formula,
\begin{equation}
\label{eq:Stir}
k!>\left(\frac{k}{e}\right)^k {\sqrt{2\pi k}}\quad {\text{\rm and}}\quad k^k<\frac{k! e^k}{\sqrt{2\pi k}}.
\end{equation}
Inserting the inequality from the right-hand side of~\eqref{eq:Stir}, we get the desired conclusion.
\end{proof}

\begin{proposition}
\label{J}
Let $k\geq 1$ and $1\leq d\leq n$. Then
\begin{equation*}
\left|K_{k,d}(n)-\Theta_{k,n}(d)\frac{(n/d)^k}{k!}\right|\le \Theta_{k,n}(d)\left(
k+\frac{e^{k+1}}{\sqrt{2\pi k}}
\right)(n/d)^{k-1}.
\end{equation*}
\end{proposition}

\begin{proof}
If $d=1$, we have
$$
K_{k,1}(n)=K_{k}(n)=\frac{n^{k}}{k!}+\theta_1 kn^{k-1}=\frac{n^{k}}{k!}+\theta_2\left(
k+\frac{e^{k+1}}{\sqrt{2\pi k}}
\right)n^{k-1}.
$$
Suppose that  $d>1$. By definition, $K_{k,d}=\#\mathcal{K}_{k,d}(n)$ and the elements of $\mathcal{K}_{k,d}(n)$  are the solutions $(x,y_1,\ldots,y_k)$ of the equation
\begin{equation}
\label{gen}
n=x+\sum_{j=1}^ky_j, \quad \textrm{ for which } d \textrm{ divides } x \textrm{ and }y_1\cdots y_k.
\end{equation}
Write $x=dX$, and $y_j=y_j^*+dY_j$ with $X>0$, $Y_j\geq 0$ and $y_j^*\in \{0,\ldots,d-1\}$, for each $j\in\{1,\ldots,k\}$. Note that, as $d$ divides $\prod_{j=1}^ky_j$, for each prime factor $p$ of $d$, there exists at least an index $j_p\in \{1,\ldots,k\}$ with $p\mid y_{j_p}^*$. Clearly, $(y_1^*,\ldots,y_k^*)$ and $(X,Y_1,\ldots,Y_k)$ are uniquely determined by $(x,y_1,\ldots,y_k)$, and similarly, the vector $(x,y_1,\ldots,y_k)$ is uniquely determined by both $(y_1^*,\ldots,y_k^*)$ and $(X,Y_1,\ldots,Y_k)$.

We now determine the number of possible tuples $(y_1^*,\ldots,y_k^*)$. Reducing ~\eqref{gen} modulo $d$, we get
\begin{equation}
\label{eq:11}
y_1^*+\cdots+y_k^*\equiv n\pmod d, \quad \textrm{ with } y_1^*\cdots y_k^*\equiv 0\pmod d.
\end{equation}
Reducing congruence~\eqref{eq:11} further modulo $p$, where $p$ is an arbitrary prime factor of $d$, we get a solution to the equation
\begin{equation}
\label{eq:gen10}
y_{1,p}^*+\cdots +y_{k,p}^*\equiv n\pmod p,\end{equation}
 with  $y_{j_p,p}^*\equiv 0\pmod p$  for at least  one $j_p\in\{1,\ldots,k\}.$

This shows that $(y_1^*,\ldots,y_k^*)$ determines a solution $(y_{1,p}^*,\ldots,y_{k,p}^*)$ of~\eqref{eq:gen10}, for each prime factor $p$ of $d$.

Conversely, for each prime factor $p$ of $d,$ let $(y_{1,p}^*,\ldots,y_{k,p}^*)\in \mathbb{Z}_p^{k}$ be a solution of~\eqref{eq:gen10}. Fix $j\in\{1,\ldots,n\}$, consider the system $y_j^*\equiv y_{j,p}^*\pmod p$ for $p\mid d$ and apply the Chinese remainder theorem to find a unique solution modulo $d.$ Looking at the equation related to $p$ in each system, we have $y_1^*+\cdots+y_k^*\equiv n\pmod p$ for all $p\mid d$ and, since $d$ is square-free, it follows that
$(y_1^*,\ldots,y_k^*)\in \mathbb{Z}_d^k$ is a solution of~\eqref{eq:11}. Note that  $y_1^*\cdots y_k^*\equiv 0\pmod d$ is a consequence of  $y_{j_p,p}^*\equiv 0\pmod p$ for at least  one $j_p\in\{1,\ldots,k\},$ because this implies that $p$ divides $y_{j_p}^*$ and consequently $d=\prod_{p\mid d}p$ divides $y_1^*\cdots y_k^*.$

Now, by Lemma~\ref{modp}~a), the number of solutions of~~\eqref{eq:gen10} is either $\phi_k(p)$ (if $p\mid n$) or $\psi_k(p)$ (if $p\nmid n$). Hence, the number of possibilities for $(y_1^*,\ldots,y_k^*)$ is
$$
\prod_{p\mid d, p\mid n}\phi_k(p)\prod_{p\mid d, p\nmid n} \psi_k(p)=\Theta_{k,n}(d).
$$
Fix $(y_1^*,\ldots,y_k^*)$ and let us determine the number of possible tuples $(X,Y_1,\ldots,Y_k)$. From ~\eqref{gen}, we  get the equation
\begin{equation}
\label{eq:uuu}
X+\sum_{j=1}^k Y_j=\frac{n-\sum_{j=1}^k y_j^*}{d},
\end{equation}
where the right-hand side  is an integer.  Recalling that $X$ and $Y_j$ are non-negative, it follows that the number of solutions of~\eqref{eq:uuu}, with respect to the natural number $m:=(n-\sum_{j=1}^ky_j^*)/d$, is between the number of $(k+1)$-compositions of $m$ and the number of generalized $(k+1)$-compositions of $m$. Thus, by~\eqref{co}, its size is
$$
\frac{m^k}{k!}+\theta_3 km^{k-1}.
$$
Since $y_j^*\in \{0,\ldots,d-1\}$, we have
$$
0\leq \sum_{j=1}^k \frac{y_j^*}{d}\leq \left(1-\frac{1}{d}\right)k\leq k.
$$
So, $|m-n/d|\le k$. Therefore, applying Lemma \ref{power} with $c=1,$ we get
$$
\left|\frac{m^k}{k!}-\frac{(n/d)^k}{k!}\right|\le \frac{e^{k+1}}{\sqrt{2\pi k}}(n/d)^{k-1}.
$$
To estimate $km^{k-1}$, we note that $m\leq n/d$ gives $km^{k-1}\le k(n/d)^{k-1}.$ Thus, for a fixed $(y_1^*,\ldots,y_k^*)$, the number of acceptable integer solutions $(X,Y_1,\ldots,Y_k)$ to equation~\eqref{eq:uuu} is
$$
\frac{(n/d)^k}{k!}+\theta_4\left[\left(k+\frac{e^{k+1}}{\sqrt{2\pi k}}\right)(n/d)^{k-1}\right].
$$
Except for the value of $\theta_4$, this does not depend on $(y_1^*,\ldots,y_k^*)$. Summing up the above expression over the  possible $(y_1^*,\ldots,y_k^*)\in \mathbb{Z}_d^k$, we get the desired result.
\end{proof}

The following elementary observation will be used in the proof of Theorem~\ref{thm:1}.

\begin{lemma}
\label{tk}
For $k\geq 0$, let
$$
I_k(x)=\int_{x}^{\infty} \frac{(\log t)^k}{t^2} dt
$$
as a  function in the real variable $x>0$.
Then $$
I_k(x)\leq 2k! \frac{(\log x)^k}{x}\quad\quad \text{for all}\quad\quad x\geq e^{4/3}.
$$
\end{lemma}

\begin{proof}
For $k=0$, we have
$$
I_0(x)=\int_{x}^\infty \frac{dt}{t^2}=\frac{1}{x}
$$
and the lemma is trivial. For $k\geq 1$, we have
\begin{eqnarray*}
I_{k-1} (x) & = & \int_{x}^{\infty} \frac{(\log t)^{k-1}}{t}\frac{dt}{t} = \frac{(\log t)^{k}}{k} \frac{1}{t}\Big|_{x}^{\infty}
+ \int_{x}^{\infty}  \frac{\log t^{k} }{k}\frac{1}{t^2} dt\\
& = &-\frac{(\log x)^{k}}{kx}+\frac{1}{k}I_k(x)
\end{eqnarray*}
which gives
$$
I_k(x)=\frac{(\log x)^k}{x}+kI_{k-1}(x).
$$
Using this relation,  the lemma follows by induction on $k\geq 1$. In fact,
$$
I_1(x)=\frac{\log x}{x}+\frac{1}{x}\leq \frac{2\log x}{x}\quad {\text{\rm for~ any}}\quad x\geq e,
$$
and so, in particular, for all $x\geq e^{4/3}.$ Moreover, by the inductive hypothesis,
\begin{equation*}
I_{k+1}(x)=\frac{(\log x)^{k+1}}{x}+(k+1)I_{k}(x)\leq \frac{(\log x)^{k+1}}{x}\left[1+\frac{2(k+1)!}{\log x}\right]\\
\end{equation*}
and, for $x\geq e^{4/3},$ we get
\begin{eqnarray*}1+\frac{2(k+1)!}{\log x}& \leq 1+\displaystyle{\frac{3(k+1)!}{2}}\leq 2(k+1)!
\end{eqnarray*}
for all  $k\geq 1.$
\end{proof}

\begin{proof}[Proof of Theorem~$\ref{thm:1}$] Due to the cases discussed in the Introduction, we can assume that $k\geq 2$ and $n\geq 4$. We begin by applying \eqref{eq:incexc} together with Proposition~\ref{J} 
\begin{eqnarray}
\label{eq:mainformula}
A_{k}(n) & = & \sum_{1\leq d\leq n} \mu(d) K_{k,d}(n)\nonumber\\
&=& \sum_{1\leq d\leq n} \mu(d)\Theta_{k,n}(d)\left(\frac{(n/d)^k}{k!}+\theta_d\left(k+\frac{e^{k+1}}{\sqrt{2\pi k}}\right)
(n/d)^{k-1}\right)
\nonumber\\
& = &   M+E
\end{eqnarray}
where
$$
M=\sum_{1\leq d\leq n} \mu(d)\Theta_{k,n}(d)\frac{(n/d)^k}{k!}
$$
is the main term and
$$
E=\sum_{1\leq d\leq n} \mu(d)\Theta_{k,n}(d)\,\theta_d\left(k+\frac{e^{k+1}}{\sqrt{2\pi k}}\right)
(n/d)^{k-1}
$$
is the error term.

Thus, by \eqref{eq:PHI}, we get
\begin{eqnarray}
\label{eq:E1}
E&=& 
\theta_1\left(k+\frac{e^{k+1}}{\sqrt{2\pi k}}\right)
n^{k-1}
\sum_{1\leq d\leq  n} \frac{\Theta_{k,n}(d)}{d^{k-1}}\\\nonumber
&=&\theta_2\left( k+\frac{e^{k+1}}{\sqrt{2\pi k}}   \right)n^{k-1} \sum_{1\leq d\leq n} \frac{k^{\omega(d)}}{d}.
\end{eqnarray}
We want to find a better estimate for $E$ through an estimate for the function:
$$
\Omega_k(x)=\sum_{1\leq d\leq x} \frac{k^{\omega(d)}}{d},
$$
defined for any real number $x\ge 1.$ From~\cite[$(3.20)$]{RS}, we have
$$
\sum_{p\leq x}\frac{1}{p}\leq \log\log x+B+\frac{1}{(\log x)^2},
$$
for each real number $x>1$, where $B$ is the Mertens' constant~\cite[$(2.10)$]{RS}. As $B\leq 0.27$, we have, in particular, that
$$
\sum_{p\leq x}\frac{1}{p}\leq \log\log x+1,\quad\quad \text{ for any }\ x\geq 4.
$$
Hence, for any real number $x\ge 4$, we have also
\begin{eqnarray}
\label{eq:calc}
\Omega_k(x)& \le &
\prod_{p\le n} \left(1+\frac{k}{p}\right)\le
\exp\left(\sum_{p\le x} \frac{k}{p}\right)\nonumber \\
& \leq & \exp\left(k(\log\log x+1)\right)=(e\log x)^k.
\end{eqnarray}
Using ~\eqref{eq:E1}, we find
\begin{equation}
\label{eq:E}
E=\theta_3 \left(k+\frac{e^{k+1}}{\sqrt{2\pi k}}\right)(e\log n)^kn^{k-1}.
\end{equation}

We now look at the main term $M$. We have
\begin{equation}
\label{eq:M}
M=\frac{n^k}{k!}\sum_{d\geq 1} \mu(d)\frac{\Theta_{k,n}(d)}{d^{k}}-\frac{n^k}{k!}\sum_{d> n}\mu(d)\frac{\Theta_{k,n}(d)}{d^k}=\frac{n^k}{k!}(M_1-M_2).
\end{equation}
We first compute $$M_1=\sum_{d\ge 1} \mu(d)\frac{\Theta_{k,n}(d)}{d^{k}}$$ and then we estimate $$M_2=\sum_{d> n} \mu(d)\frac{\Theta_{k,n}(d)}{d^{k}}.$$
For $M_1$ we have, by the multiplicativity of $\Theta_{k,n}(d)$ as a function of $d$,
\begin{eqnarray}
\label{eq:M1}
M_1 & = & \prod_{p\mid n} \left(1-\frac{\phi_k(p)}{p^k}\right)\prod_{p \nmid n}
 \left(1-\frac{\psi_k(p)}{p^k}\right)\nonumber\\
 & = &
\prod_{p\mid n}
\left(1-\frac{\phi_k(p)}{p^k}\right)
\left(1-\frac{\psi_k(p)}{p^{k}}\right)^{-1}
\prod_{p}
 \left(1-\frac{\psi_k(p)}{p^k}\right)\nonumber\\
 & = & C_k f_k(n),
 \end{eqnarray}
where the last equality arises upon observing that, by \eqref{etadelta},
 $$
\left(1-\frac{\phi_k(p)}{p^k}\right)\left(1-\frac{\psi_k(p)}{p^{k}}\right)^{-1}=
\frac{p^k-\phi_k(p)}{p^k-\psi_k(p)}=1+\frac{(-1)^k}{p^k-\psi_k(p)}.
$$
For $M_2$, we use~\eqref{eq:PHI} to conclude that
$$
|M_2|\le  \sum_{d> n} \frac{k^{\omega(d)}}{d^2}.
$$
By the Euler summation formula on $\Omega_k(x)$ (see Theorem 3.1 in \cite{Apostol}), we get
$$
\sum_{x< d\le X} \frac{k^{\omega(d)}}{d^2}=
\int_x^X \frac{(\Omega_k(t))'}{t}dt =\frac{\Omega_k(t)}{t}\Big|_{x}^X+\int_{x}^X \frac{\Omega_k(t)}{t^2}dt,
$$
for any $X>x.$ Since ~\eqref{eq:calc} implies that  $\Omega_k(t)=O((e\log t)^k)$, taking $X\to\infty,$ the first summand is equal to $-\Omega_k(x)/x$ and, in particular, is negative. Therefore, using again ~\eqref{eq:calc}, we deduce that for $x\geq 4$:
$$
\sum_{x< d} \frac{k^{\omega(d)}}{d^2}\le \int_x^\infty \frac{\Omega_k(t)}{t^2}dt\leq e^k \int_{x}^{\infty} \frac{(\log t)^k}{t^2} dt.
$$
Hence, from Lemma~\ref{tk}, we obtain
\begin{equation}
\label{eq:M2}
|M_2|\leq \frac{2k! e^k (\log n)^k }{n}.
\end{equation}
Since, for $k\geq 2$, we have $$\left(k+\frac{e^{k+1}}{\sqrt{2\pi k}}\right)e^k+2e^k= \left(k+2+\frac{e^{k+1}}{\sqrt{2\pi k}}\right)e^k<\frac{(2+e)e^{2k}}{\sqrt{2\pi k}},$$  
the desired conclusion follows finally from \eqref{eq:mainformula},~\eqref{eq:E},~\eqref{eq:M},~\eqref{eq:M1} and~\eqref{eq:M2}.
\end{proof}

\section{Proof of Theorem~\ref{thm:allcoprime}}\label{sec:allcoprime}

We start with two definitions and a proposition which play a role similar to Definitions~\ref{def:kdn},~\ref{theta} and Proposition~\ref{J}.

\begin{definition}
\label{def:B2k}
Given positive integers $k$ and $n$, write $\mathcal{B}_{k,d}(n)$ for the set of $k$-compositions $(x_1,\ldots,x_k)$ of $n$ such that $\gcd(x_i,x_j)$ is coprime to $d$, for every distinct $i,j\in \{1,\ldots,k\}$, that is,
\begin{equation*}
\begin{aligned}
\mathcal{B}_{k,d}(n)=\{(x_1,\ldots,x_k)\in \mathbb{N}^{k}:\ & n=x_1+\cdots +x_k \textrm{ and } p\nmid \gcd(x_i,x_j), \textrm{ for each}\\
& \textrm{prime }p\mid d \textrm{ and for distinct }i,j\in \{1,\ldots,k\}\}.
\end{aligned}
\end{equation*} and put $B_{k,d}(n)=\#\mathcal{B}_{k,d}(n).$
\end{definition}

Clearly $\mathcal{B}_{k,1}(n)=\mathcal{K}_{k-1}(n)$ and  $\mathcal{B}_k(n)\subseteq \mathcal{B}_{k,d}(n)$.

\begin{definition}
\label{Xi}
For positive integers $k$ and $n$,
write
$$
\Xi_{k,n}(d)=\prod_{p\mid d, p\mid n}\eta_k(p) \prod_{p\mid d, p\nmid n}\delta_k(p),
$$
where for $d=1$ the empty product is taken to be 1. \end{definition}
Note that, by Lemma~\ref{modp}~ e) and f), we have
\begin{equation}
\label{eq:Xi}
1\leq \Xi_{k,n}(d)\le  d^{k-1}.
\end{equation}

\begin{proposition}
\label{JJ}
Let $k\geq 3$ and $1\leq d\leq n$. Then
$$
\left|B_{k,d}(n)-\Xi_{k,n}(d)
\frac{(n/d)^{k-1}}{(k-1)!}\right|\le \Xi_{k,n}(d)\left( k-1+\frac{e^{2/3}(3e/2)^{k-1}}{\sqrt{2\pi(k-1)}} \right)n^{k-2}d.
$$
\end{proposition}

\begin{proof}
The proof of this proposition is very similar to the proof of Proposition~\ref{J}. 
For $d=1$ the statement is trivial because, by \eqref{co},
$$
B_{k,1}(n)=K_{k-1}(n)=\frac{n^{k-1}}{(k-1)!}
+\theta_1(k-1)n^{k-2}.
$$
Let $d>1.$ By Definition~\ref{def:B2k}, the elements  $(x_1,\ldots,x_k)$ of $\mathcal{B}_{k,d}(n)$ are the solutions  of the equation
\begin{equation}\label{gen2}
n=x_1+\cdots +x_k \ \textrm{for which } d \textrm{ is coprime to }\gcd(x_i,x_j), \textrm{ for every }i<j.
\end{equation}
Write $x_i=dX_i+x_{i}^*$ with $X_i\geq 0$ and $x_i^*\in \mathbb{Z}_d$, for each $i\in \{1,\ldots,k\}.$
First we examine the possible tuples  $(x_1^*,\ldots,x_k^*)\in \mathbb{Z}_d^{k}$ which can arise.

Reducing ~\eqref{gen2} modulo $d$, we get
\begin{equation}
\label{eq:112}
x_1^*+\cdots+x_k^*\equiv n\pmod d \quad \textrm{with }p\nmid \gcd(x_i^*,x_j^*), \textrm{ for each }p\mid d \textrm{ and }i\neq j,
\end{equation}
and thus  $(x_1^*,\ldots,x_k^*)$ belongs to the set
$$\begin{aligned}
\mathcal{S}=\{(x_1^*,\ldots,x_k^*)\in \mathbb{Z}_d^{k}\ :&\  x_1^*+\cdots+x_k^*\equiv n\pmod d \quad \textrm{with }p\nmid \gcd(x_i^*,x_j^*), \\ & \textrm{ for each }p\mid d \textrm{ and }i\neq j\}.
\end{aligned}
$$

Reducing~\eqref{eq:112} further modulo $p$, where $p$ is an arbitrary prime factor of $d$, we get a unique solution to the  equation
\begin{equation}
\label{eq:gen102}
x_{1,p}^*+\cdots+x_{k,p}^*\equiv n\pmod p \quad \textrm{with }x_{i,p}^*\in  \mathbb{Z}_p\  \textrm{and } x_{i,p}^*= 0\textrm{ for at most one } i.
\end{equation}
Note that, by Lemma~\ref{modp}~d), the number of solutions of ~\eqref{eq:gen102} is either $\eta_k(p)$ (if $p\mid n$) or $\delta_k(p)$ (if $p\nmid n$). Moreover, if we consider a solution $(x_{1,p}^*,\ldots,x_{k,p}^*)\in\mathbb{Z}_p^{k}$ of~\eqref{eq:gen102} for any prime factor $p$ of $d$ and apply the Chinese remainder theorem  in each one of the $k$ coordinates, we get a unique solution $(x_{1}^*,\ldots,x_{k}^*)\in\mathbb{Z}_d^{k}$ of~\eqref{eq:112} with $x_{i}^*\equiv x_{i,p}^*\pmod p$, for all $i\in \{1,\ldots,k\}$ and all $p\mid d.$
Hence,
\begin{equation}
\label{eq:prod}
\#\mathcal{S}=\prod_{p\mid d,p\mid n}\eta_k(p)\prod_{p\mid d, p\nmid n}\delta_k(p)=\Xi_{k,n}(d).
\end{equation}

Now we fix  $x^*=(x_1^*,\ldots,x_k^*)\in\mathcal{S}$  and we determine the possible tuples $X=(X_1,\ldots,X_k)$
which are compatible with~\eqref{gen2}, that is, the solutions  $\mathcal{T}_{x^*}$ of
\begin{equation*}\label{eq:uuu2}
X_1+\cdots+X_k=\frac{(n-\sum_{i=1}^k x_{i}^*)}{d},
\end{equation*}
where the right-hand side $m=(n-\sum_{i=1}^{k}x_i^*)/d$ is an integer.
Clearly there is a bijection between  $\mathcal{B}_{k,d}(n)$ and
$$
\{(x^*,X):x^*\in  \mathcal{S}, X\in \mathcal{T}_{x^*}\}.
$$
So, to estimate $B_{k,d}(n)$ we just need to estimate $\#\mathcal{T}_{x^*}$, for every  $x^*\in \mathcal{S}.$

 Recalling that $X_i$ is non-negative, it follows from ~\eqref{co} that the number of solutions of~\eqref{eq:uuu2}, with respect to  $m$, is
 $$
 \#\mathcal{T}_{x^*}=\frac{m^{k-1}}{(k-1)!}  +  \theta_2(k-1)m^{k-2}.
 $$
As $m\leq n/d$, we have $(k-1)m^{k-2}\le (k-1)(n/d)^{k-2}$. Moreover, since $x_{i}^*<d,$ we have also $m=n/d+\theta_3k$. As $k\geq 3$, we obtain  $m=n/d+(3/2)\theta_4(k-1)$. Since $n/d\geq 1$, Lemma~\ref{power} applies with $c=3/2$ giving:
$$
\frac{m^{k-1}}{(k-1)!}=\frac{(n/d)^{k-1}}{(k-1)!}+\theta_5\frac{e^{2/3}(3e/2)^{k-1}}{\sqrt{2\pi(k-1)}}(n/d)^{k-2}.
$$
It follows that
\begin{equation}\label{eq:ult}
\#\mathcal{T}_{x^*}=\frac{(n/d)^{k-1}}{(k-1)!}+\theta_6\left(k-1+\frac{e^{2/3}(3e/2)^{k-1}}{\sqrt{2\pi(k-1)}}\right)(n/d)^{k-2}.
\end{equation}
Now, multiplying \eqref{eq:prod} and~\eqref{eq:ult} together  and  using  \eqref{eq:Xi}, we get finally
\begin{equation*}
\begin{aligned}
B_{k,d}(n)&= \Xi_{k,n}(d)\frac{(n/d)^{k-1}}{(k-1)!}+\Xi_{k,n}(d)\,\theta_7\left(k-1+\frac{e^{2/3}(3e/2)^{k-1}}{\sqrt{2\pi(k-1)}}\right)(n/d)^{k-2}\\
&= \Xi_{k,n}(d)\frac{(n/d)^{k-1}}{(k-1)!}+\theta_8\left(k-1+\frac{e^{2/3}(3e/2)^{k-1}}{\sqrt{2\pi(k-1)}}\right)n^{k-2}d.
\end{aligned}
\end{equation*}
\end{proof}

\begin{proof}[Proof of Theorem~\ref{thm:allcoprime}]
Due to the cases discussed in the Introduction, we may assume that $k\geq 3.$ Let $n\geq e^{k2^{k+2}} $ and write
$$
q(n)=\frac{\log n}{2}\quad {\text{\rm and}}\quad d(n)=\prod_{p\le q(n)} p.
$$
From \cite[Theorem~$6$]{RS1}, we have
$$
\sum_{p\le x} \log p<1.001102 x\quad {\text{\rm for~all}}\quad x>1.
$$
Therefore
\begin{equation}
\label{d:bound}
d(n)\leq \exp(1.001102 q(n))\leq n^{1.001102/2}<n^{0.5006}.
\end{equation}
Let
 \begin{equation*}
 \begin{aligned}
 {\mathcal B}'_{k}(n)= \{(x_1,\ldots,x_k)\in \mathcal{K}_{k-1}\ :&\  \text{there exist} \  i, j\in \{1,\ldots,k\} \textrm{ with }i\neq j, \textrm{ and }p\\\nonumber & \text{ with } p>q(n) \textrm{ and }\ p\mid \gcd(x_i,x_j)\}\end{aligned}\end{equation*}
and $B'_{k}(n)=\# \mathcal{B}'_k(n).$ Using Definition~\ref{def:B2k}, we have
$$
\mathcal{B}_k(n)=\mathcal{B}_{k,\,d(n)}(n) \setminus{\mathcal B}'_{k}(n),
$$
and, in particular,
\begin{equation}\label{end}
{B}_k(n)=B_{k,\,d(n)}(n)+ \theta_1 B'_{k}(n).
\end{equation}
Estimating both $B_{k,\,d(n)}(n)$ and $B'_{k}(n)$, we will see that the main part of ${B}_k(n)$ is given by $B_{k,\,d(n)}(n)$. 

First, we claim that
\begin{equation}
\label{B1k}
B'_k(n)=\theta_2\frac{24n^{k-1}}{\log n}.
\end{equation}
To get an element of ${\mathcal B}'_{k}(n)$, the pair $\{i,j\}$ can be chosen in  $\binom{k}{2}$ ways. Once the pair $\{i,j\}$ is chosen and the prime $p>q(n)$ is fixed, we see that  $x_i$ and $x_j$ are both multiples of $p$ of magnitude at most $n$. Thus, the ordered pair $(x_i,x_j)$ can be chosen in at most $(n/p)^2$ ways. Once the pair $(x_i,x_j)$ is chosen, we have
$$
\sum_{1\le \ell\le  k, \ell\not\in \{i,j\}} x_{\ell}=n-(x_i+x_j).
$$
Therefore, the number of choices for the remaining summands $x_{\ell}$ is the number of $(k-2)$-compositions of $n-(x_i+x_j)$, that is,
$$
\binom{n-(x_i+x_j)-1}{k-3}\le \frac{n^{k-3}}{(k-3)!}.
$$
Summing up, we obtain
$$
 B'_{k}(n)\le \sum_{p>q(n)} \binom{k}{2}\frac{n^2}{p^2}\frac{n^{k-3}}{(k-3)!}\le \frac{n^{k-1} k(k-1)}{2(k-3)!} \sum_{p>q(n)}\frac{1}{p^2}.
$$
Observe now that
\begin{eqnarray}\nonumber
\sum_{p>q(n)} \frac{1}{p^2}  &\leq& \frac{1}{q(n)^2}+\int_{q(n)}^{\infty} \frac{dt}{t^2}
=\frac{1}{q(n)^2}+\left(-\frac{1}{t}\right)\Big|_{q(n)}^{\infty}\\ \label{4logn}
&=&\frac{1}{q(n)^2}+\frac{1}{q(n)}\le \frac{4}{\log n}.
\end{eqnarray}
This gives
\begin{equation*}
\label{eq:B1}
 B'_{k}(n)\le \frac{4 k(k-1)}{2(k-3)!}\frac{n^{k-1}} {\log n}\le \frac{24n^{k-1}}{\log n},
\end{equation*}
where for the last inequality we used the fact that
$$
\frac{2k(k-1)}{(k-3)!}\leq 24\quad {\text{\rm for~ all}}\quad k\geq 3,
$$
which proves~\eqref{B1k}.

Next, we estimate $B_{k,d(n)}(n) .$
Using Proposition \ref{JJ}, we have
\begin{equation}
\label{eq:cardB2}
 B_{k,d(n)}(n)  =\Xi_{k,n}(d(n))
\frac{(n/d(n))^{k-1}}{(k-1)!} +\theta_3\left( k-1+\frac{e^{2/3}(3e/2)^{k-1}}{\sqrt{2\pi(k-1)}} \right)n^{k-2}d(n).
\end{equation}
Extending the product from the main term $M$ of ~\eqref{eq:cardB2} to all primes, we get
\begin{eqnarray}
\label{eq:last}\nonumber
M&=&\frac{n^{k-1}}{(k-1)!}\frac{\Xi_{k,n}(d(n))}{d(n)^{k-1}}=\frac{n^{k-1}}{(k-1)!}\,\prod_{p\mid n, p\leq q(n)}\frac{\eta_k(p)}{p^{k-1}}\,\prod_{p\nmid n, p\leq q(n)}\frac{\delta_k(p)}{p^{k-1}}\\\nonumber
& = &\frac{n^{k-1}}{(k-1)!}\prod_{p\mid n}\frac{\eta_k(p)}{p^{k-1}}\hspace{-0.1cm}\prod_{p\mid n, p>q(n)}\!\frac{p^{k-1}}{\eta_k(p)}\prod_{p}\frac{\delta_k(p)}{p^{k-1}}\prod_{p\mid n}\frac{p^{k-1}}{\delta_k(p)}\hspace{-0.1cm}\prod_{p\nmid n,p>q(n)}\frac{p^{k-1}}{\delta_k(p)}\\\nonumber
& = &D_k g_k(n)\frac{n^{k-1}}{(k-1)!}\, \prod_{p\mid n, p>q(n)}
 \frac{p^{k-1}}{\eta_k(p)}\, \prod_{p\nmid n, p>q(n)} \frac{p^{k-1}}{\delta_k(p)}\\
& = & D_kg_k(n)\frac{n^{k-1}}{(k-1)!}E_k(n),
\end{eqnarray}
(where in the third equality we used the relation~\eqref{etadelta} between  $\eta_k$ and $\delta_k$).
We now estimate the error term
$$
E_k(n)=\prod_{p\mid n, p>q(n)}
\frac{p^{k-1}}{\eta_k(p)} \prod_{p\nmid n, p>q(n)} \frac{p^{k-1}}{\delta_k(p)}.
$$

Since $q(n)\geq k2^{k+1}>\sqrt{2^kk}$, from Lemma~\ref{modp}~f) and g), we get that
$$
0\leq \max\left\{\log\left(\frac{p^{k-1}}{\eta_k(p)}\right),\ \log\left(\frac{p^{k-1}}{\delta_k(p)}\right)\right\}\leq \log\left(1+\frac{k2^{k}}{p^2}\right)\leq \frac{k2^k}{p^2},$$
for any $p>q(n).$ In particular, $E_k(n)\geq 1$ and using ~\eqref{4logn}, we get
$$
0\leq\log (E_k(n))\leq k2^{k}\sum_{p>q(n)} \frac{1}{p^2}\leq \frac{2^{k+2} k}{\log n}\leq 1.
$$
 Recalling that $e^y\leq 1+2y$ for any $0\leq y\leq 1,$  we reach finally
\begin{equation}\label{E bound}
E_k(n)=1+\theta_4\frac{2^{k+3} k}{\log n}.
\end{equation}

Now we go back to the main term $M$. By
 Theorem~\ref{thm:3}, $1/(2k)<g_k(n)<2k$ and $D_k< 1.$ Then, using \eqref{eq:last} and~\eqref{E bound}, we obtain
\begin{eqnarray}
\label{eq:mainB2}
M&=&
D_k g_k(n)\frac{n^{k-1}}{(k-1)!}\left(1+\theta_4\frac{2^{k+3} k}{\log n}\right)\nonumber\\
& = &D_k g_k(n)\frac{n^{k-1}}{(k-1)!} + \theta_5\frac{2^{k+4} k^2}{(k-1)!} \frac{n^{k-1}}{\log n}\nonumber\\
& = &D_k g_k(n)\frac{n^{k-1}}{(k-1)!}+682.7\theta_6\frac{n^{k-1}}{\log n},
\end{eqnarray}
where we used the fact that
$$
\frac{2^{k+4}k^2}{(k-1)!}\leq 682.7\quad {\text{\rm for~all}}\quad k\geq 3.
$$
We now estimate the error in ~\eqref{eq:cardB2}.  First of all observe that since $\log n\leq n^{\alpha}$  holds for any $\alpha\geq e^{-1},$ we surely have $\log n\leq n^{0.3994}$ and thus, using \eqref{d:bound},
\begin{eqnarray*}
n^{k-2}d(n) \leq  n^{k-1}n^{-0.4994}=
\frac{1}{n^{0.1}}\frac{\log n}{n^{0.3994}}\frac{n^{k-1}}{\log n}\leq \frac{1}{n^{0.1}}\frac{n^{k-1}}{\log n}.
\end{eqnarray*}
Since $n\geq e^{k2^{k+2}},$ we have
$$
\left( k-1+\frac{e^{2/3}(3e/2)^{k-1}}{\sqrt{2\pi(k-1)}} \right)\frac{1}{n^{0.1}}\leq \frac{k-1+\frac{e^{2/3}(3e/2)^{k-1}}{\sqrt{2\pi(k-1)}}}{e^{\frac{k2^{k+2}}{10}}},
$$
which is a decreasing function of $k$ whose values is always strictly less than $0.002.$

This says that  the error in~\eqref{eq:cardB2}  can be written as $0.002\theta_7{n^{k-1}}/{\log n}$.
Summing up,  considering~\eqref{end},~\eqref{B1k} and~\eqref{eq:mainB2} and noting that $682.7+24+0.002<707$, we find 
$$
B_k(n)=D_kg_k(n)\frac{n^{k-1}}{(k-1)!}+707 \theta_8 \frac{n^{k-1}}{\log n},
$$
which is what we wanted.
\end{proof}

\end{document}